\documentclass[11pt]{article}
\usepackage{latexsym}  
\usepackage{amsmath,amssymb,amsfonts}
\usepackage{enumerate}

\begin{document}

\title{On DR-semigroups satisfying the ample conditions}
\author{Tim Stokes}

\date{}
\maketitle

\newcommand{\bea}{\begin{eqnarray*}}
\newcommand{\eea}{\end{eqnarray*}}

\newcommand{\ben}{\begin{enumerate}}
\newcommand{\een}{\end{enumerate}}

\newcommand{\bi}{\begin{itemize}}
\newcommand{\ei}{\end{itemize}}

\newcommand{\beq}{\begin{equation}}
\newcommand{\eeq}{\end{equation}}

\newenvironment{proof}{\noindent \textbf{Proof.}\hspace{.7em}}
                   {\hfill $\Box$
                    \vspace{10pt}}

\newcommand{\mc}{\mathcal}

\newtheorem{thm}{Theorem}[section]
\newtheorem{pro}[thm]{Proposition}
\newtheorem{lem}[thm]{Lemma}
\newtheorem{cor}[thm]{Corollary}
\newtheorem{example}{Example}

\newcommand{\dom}{\mbox{dom}}
\newcommand{\im}{\mbox{im}}

\begin{abstract}
A DR-semigroup $S$ (also known as a reduced E-semiabundant or reduced E-Fountain semigroup) is here viewed as a semigroup equipped with two unary operations $D,R$ satisfying finitely many equational laws.  Examples include DRC-semigroups (hence Ehresmann semigroups), which also satisfy the congruence conditions.  The ample conditions on DR-semigroups are studied here and are defined by the laws 
$$xD(y)=D(xD(y))x\mbox{ and }R(y)x=xR(R(y)x).$$  Two natural partial orders may be defined on a DR-semigroup and we show that the ample conditions hold if and only if the two orders are equal and the projections (elements of the form $D(x)$) commute with one-another.  Restriction semigroups satisfy the generalized ample conditions, but we give other examples using strongly order-preserving functions on a quasiordered set as well as  so-called ``double demonic" composition on binary relations.  Following the work of Stein, we show how to construct a certain partial algebra $C(S)$ from any DR-semigroup, which is a category if $S$ satisfies the congruence conditions, but is ``almost" a category if the ample conditions hold.  We then characterise the ample conditions in terms of a converse of the condition on $S$ ensuring that $C(S)$ is a category.  Our main result is an ESN-style theorem for DR-semigroups satisfying the ample conditions, based on the $C(S)$ construction.  We also obtain an embedding theorem, generalizing a result for restriction semigroups due to Lawson.
\end{abstract}

\medskip


\section{Preliminaries}

\subsection{Partial algebras}
 
First, some terminology for partial operations.  Let $S$ be a set with partial binary operation $\circ$.  We here say that $(S,\circ)$ is a {\em partial semigroup} if $x\circ(y\circ z)$ is defined if and only if $(x\circ y)\circ z$ is, and then they are equal; equivalently (as is well-known), it is a semigroup with zero, made into a partial algebra by removing $0$ and retaining all nonzero products. In a partial semigroup, we may unambiguously write $a\circ b\circ c$ (if it exists).  A number of other definitions of partial semigroups exist in the literature, but the one used here is perhaps the most commonly used one.

Here we define a {\em (small) partial category} $(C,\circ,D,R)$ as follows: it is a set $C$ equipped with a partial binary operation $\circ$ and two unary operations of ``domain" $D$  and ``range" $R$ such that
\begin{enumerate}[ (PC1)]
\item for all $x\in C$, $D(x)\circ x$ and $x\circ R(x)$ both exist and equal $x$
\item for all $x,y\in C$, if $x\circ y$ exists then $R(x)=D(y)$
\item $(C,\circ)$ is a partial semigroup.
\end{enumerate}
It is a {\em category} if it satisfies the following strengthening of Law (PC2) above:
\begin{itemize}
\item for all $x,y\in C$, $x\circ y$ exists if and only if $R(x)=D(y)$.
\end{itemize}

It follows easily that in the partial category $C$,
\begin{itemize}
\item for all $x\in C$, $R(D(x))=D(x)$, $D(R(x))=R(x)$, and
\item for all $x,y\in C$, if $x\circ y$ exists then $D(x\circ y)=D(x), R(x\circ y)=R(y)$.
\end{itemize}
This is because, for all $x\in C$, $D(x)\circ x$ exists and so $R(D(x))=D(x)$, and so if $x\circ y$ exists then so does $D(x\circ y)\circ (x\circ y)$ whence so must $D(x\circ y)\circ x$, and so $D(x\circ y)=R(D(x\circ y))=D(x)$; dually, $D(R(x))=R(x)$ and $R(x\circ y)=R(y)$.

Such an ``object-free" formulation of the concept of a category is frequently used in algebra; see for example \cite{FitzKin21} and \cite{Lawson91}.  
The more general partial category notion is due to Stein, who was led to define the notion of a {\em category with partial composition} in \cite{Stein24}.  Those are the ``with objects" form of the partial categories just defined.  

We shall see that partial categories arise naturally from DR-ample semigroups, but can also be used to construct them.  Examples are also easily obtained as subsets of small categories.
  
\begin{example} \label{eg:sat}
Suppose $(C,\circ,D,R)$ is a small category, with $P\subseteq C$.   We borrow terminology from ring theory and say that $P$ is {\em saturated} if for all $a,b\in C$ for which $a\circ b$ exists, $a\circ b\in P$ implies that both $a,b\in P$.  Then for saturated $P$, we define, for all $s,t\in P$, $s\cdot t=s\circ t$ providing $s\circ t\in P$, and undefined otherwise.

\begin{pro} If $C$ is a small category and $P\subseteq C$ is saturated, then $(P,\cdot,D,R)$ is a partial category.
\end{pro}
\begin{proof}
First note that if $s\in P$ then because $D(s)\circ s=s$, also $D(s)\in 
P$, and similarly, $R(s)\in P$. So $D(s)\cdot s=s=s\cdot R(s)$.  Of course, if $x,y\in P$ and $x\cdot y$ exists in $P$, then so does $x\circ y$ in $C$, and so $R(x)=D(y)$.  Finally, if $x,y,z\in P$ and $(x\cdot y)\cdot z$ exists in $P$, then $x\circ (y\circ z)=(x\circ y)\circ z$ exists in $C$ and lies in $P$, so $y\circ z\in P$, and so $y\cdot z$ exists in $P$, and similarly for $x\cdot(y\cdot z)$ which equals $x\circ (y\circ z)=(x\circ y)\circ z=(x\cdot y)\cdot z$.
\end{proof}
\end{example}

A rich family of examples arising in this way arises by letting $C$ be the small posetal category $C$ obtained from the preordered set $(X,\leq)$.  In the current object-free formulation, we can take $C=\{(x,y)\in X\times X\mid x\leq y\}$, with $(x,y)\circ (u,v)=(x,v)$ if and only if $y=u$ and otherwise it is undefined, with $D((x,y))=(x,x)$ and $R((x,y))=(y,y)$.  Then it is easily checked that $P\subseteq C$ is saturated if and only if, whenever $(x,z)\in P$ with $x\leq y\leq z$, it must be that $(x,y),(y,z)\in P$.  For example, let $X$ be the totally ordered set of integers and define $C=\{(x,y)\in X\times X\mid x\leq y\}$ as above. If we let $P_n=\{(x,y)\in C\mid y-x<n\}$ for some positive integer $n$, $P_n$ is evidently saturated, and so defining $(x,y)\cdot(y,z)=(x,z)$ but only if $z-x<n$ and otherwise undefined makes $(P_n,\cdot,D,R)$ into a partial category which is evidently not a category. For example, $(2,4),(4,6)\in P_3$ and $(2,4)\circ (4,6)=(2,6)$ in $C$ but the result is not in $P_3$, so $(2,4)\cdot (4,6)$ does not exist, even though $R((2,4))=(4,4)=D((4,6))$.

Another kind of example comes from within the path category (free category) $C$ of a directed graph, if we let $C_n$ consist of all paths of length at most $n>0$; clearly $C_n$ is saturated within $C$, but $C_n$ is not a subcategory of $C$.

\subsection{DR-semigroups}

The term ``DR-semigroup" was introduced in \cite{StokesDR15}, but the concept was first defined in \cite{Lawson91}, in the form of so-called reduced $E$-semiabundant semigroups.  These are semigroups with distinguished subset of idempotents $E$ such that the induced generalized Green's relations  $\tilde{\mathcal R}_E$ and $\tilde{\mathcal L}_E$ are such that each $\tilde{\mathcal R}_E$-class and each $\tilde{\mathcal L}_E$-class contains a member of $E$; it is further assumed that for all $e,f\in E$, $e=ef\; \Leftrightarrow\; e=fe$, and this is sufficient (but not necessary) to ensure that the member of $E$ in each $\tilde{\mathcal R}_E$-class and each $\tilde{\mathcal L}_E$-class is unique.  This gives rise to two unary operations as follows: for all $s\in S$, $D(s)$ is the unique $e\in E$ for which $(s,e)\in \tilde{\mathcal R}_E$, and similarly for $R(s)$ in terms of $\tilde{\mathcal L}_E$.  One can then equationally characterise the resulting biunary semigroups as those satisfying the following laws:
\begin{enumerate}[ (DR1)]
\item $D(x)x=x$
\item $xR(x)=x$
\item $R(D(x))=D(x), D(R(x))=R(x)$
\item $D(x)D(xy)=D(xy)D(x)=D(xy)$
\item $R(y)R(xy)=R(xy)R(y)=R(xy)$.
\end{enumerate}
These are the DR-semigroups of \cite{StokesDR15}.  They also appear in the guise of ``reduced $E$-Fountain semigroups" in \cite{Stein24} (as well as in earlier papers in the series in which the congruence conditions were also assumed).   We call the subset $D(S)=\{D(s)\mid s\in S\}$ of the DR-semigroup $S$ the set of {\em projections} of $S$.  Because of the left-right symmetry in their definition, we may dualise established facts to obtain new ones (by reversing product orders and interchanging $D$ and $R$), and often do so in what follows.

The laws 
$$D(xy)=D(xD(y))\mbox{ and }R(xy)=R(R(x)y)$$ 
are referred to as the {\em congruence conditions} and define the subclass of DR-semigroups called {\em DRC-semigroups}.   The congruence conditions ensure that the DR-semigroup $S$ can be made into a category $(S,\circ,D,R)$ if we define $a\circ b=ab$ but only when $R(a)=D(b)$.  However, they are not necessary for this: as discussed in \cite{StokesESN23}, a necessary and sufficient condition is that the following condition (equivalent to two quasiequations) holds:
$$R(x)=D(y) \; \Rightarrow \; D(xy)=D(x)\: \&\: R(xy)=R(y).$$
We call this the {\em cat-semigroup condition}.

We summarise some familiar key properties of DR-semigroups in the following result.

\begin{lem}  \label{DRbits}
Let $S$ be a DR-semigroup with $\leq$ the natural order on $D(S)$ (given by $e\leq f$ if and only if $e=ef=fe$).
\begin{enumerate}
\item For $x\in S$, $D(x)$ is the smallest $e\in D(S)$ under $\leq$ such that $ex=x$, and dually for $R(x)$.
\item For $x,y\in S$, $D(xy)\leq D(xD(y))\leq D(x), R(xy)\leq R(R(x)y)\leq R(y)$.
\end{enumerate}
\end{lem}  

Indeed, an alternative formulation of DR-semigroups, given in \cite{StokesDR15}, involves the first property above: they are precisely those semigroups having distinguished subset of idempotents $E$ such that there is a smallest $e\in E$ under the natural order such that $es=s$, namely $D(s)$, and dually for $R(s)$ (and then $E=D(S)=R(S)$). 

A further useful pair of laws is the following.

\begin{lem}  \label{lemuse}
Let $S$ be a DR-semigroup.  Then for all $x,y\in S$, $D(xy)=D(D(xy)x)$ and $R(xy)=R(yR(xy))$.
\end{lem}
\begin{proof}
Now $D(xy)=D(D(xy)xy)\leq D(D(xy)x)\leq D(xy)$ from Lemma \ref{DRbits}, so all are equal and so $D(xy)=D(D(xy)x)$.  Dualise for the other claim.
\end{proof}

Let $S$ be a DR-semigroup.  If $ef=fe$ for all $e,f\in D(S)$, then we say that {\em $D(S)$ commutes}, and $D(S)$ is easily seen to be a semilattice in this case.  (If $D(S)$ commutes and $e,f\in D(S)$, then $D(ef)=D(ef)D(e)=D(ef)e=D(fe)e=D(fe)fe=fe=ef$, so $ef=fe\in D(S)$.)  

If $D(S)$ commutes and $S$ is a DRC-semigroup, then $S$ is said to be an {\em Ehresmann semigroup}; these were first defined in \cite{Lawson91}.  Every inverse semigroup is an Ehresmann semigroup if we define $D(x)=xx'$ and $R(x)=x'x$ for all $x\in S$ (where $x'$ is the unique inverse of $x$).  

There are two well-known partial orders on any DR-semigroup $S$, both of which extend the natural order on $D(S)$, defined as follows: 
$$s\leq_r t\ \Leftrightarrow\ D(s)\leq D(t)\: \&\: s=D(s)t,$$
$$s\leq_l t\ \Leftrightarrow\ R(s)\leq R(t)\: \&\: s=tR(s).$$
Equivalently, $s\leq_r t$ if and only if $s=et$ for some $e\in D(S), e\leq D(t)$ (under the natural order on $D(S)$), and dually $s\leq_l t$ if and only if $s=tf$ for some $f\in D(S), f\leq R(t)$.  These relations were defined in \cite{Lawson91}, where they were shown to be partial orders, and they have played a key role in developing ESN-style theorems for Ehresmann and, more generally, DRC-semigroups.  

Let $C$ be a partial category such that $D(C)$ is partially ordered, and let $P(C)$ consist of all subsets of $C$.  (Evidently, $C$ could be a category, or the partial order on $D(C)$ could be equality.)  Because $C$ is a partial semigroup, $P(C)$ can obviously be made into a semigroup by defining, for all $A,B\in P(C)$, 
$$AB=\{a\circ b\mid a\in A,b\in B\}.$$
Also define $D(A)=\{e\in D(C)\mid e\leq D(a), a\in A\}$ for all $A\in P(C)$, and dually for $R(A)$.  

It is shown in Example 1 in \cite{Lawson21} that if $C$ is a small category and the partial order on $D(C)$ is taken to be equality, then $(P(C),\cdot,D,R)$ is an Ehresmann semigroup, generalizing the fact that the set of binary relations on a set gives an Ehresmann semigroup under relational composition, domain and range.  More generally, we have the following.

\begin{pro}
If $C$ is a partial category in which $D(C)$ is partially ordered, then $S=(P(C),\cdot,D,R)$ is a DR-semigroup in which products in $D(S)$ are intersections; hence, $D(S)$ commutes.
\end{pro}
\begin{proof}
Certainly $(P(C),\cdot)$ is a semigroup.

Next, for $A\in P(C)$, if $a\in A$ then $a=D(a)\circ a\in D(A)A$, so $A\subseteq D(A)A$.  Conversely, if $e\circ a\in D(A)A$ (so $e\in D(A),a\in A$), then $e=R(e)=D(a)$ and so $e\circ a=D(a)\circ a=a\in A$, and so in fact $A=D(A)A$.  Dually, $A=AR(A)$.
Evidently $D(R(A))=R(A), R(D(A))=D(A)$ for all $A\in P(C)$.  

Pick $A,B\in P(C)$.  If $e\circ f\in D(A)D(B)$ (where $e\in D(A), f\in D(B)$), then $e=f$, and so $e\circ f=e=f\in D(A)\cap D(B)$.  Conversely, if $e\in D(A)\cap D(B)$ then $e=e\circ e\in D(A)D(B)$.  So $D(A)\cap D(B)=D(A)D(B)$, which thus equals $D(B)D(A)$ also.  Hence $D(A)\subseteq D(B)\;\Leftrightarrow\; D(A)=D(A)D(B)=D(B)D(A)$.  

Thus, for $e\in D(AB)$, $e\leq D(a\circ b)=D(a)$ for some $a\in A,b\in B$, so $e\in D(A)$, and so $D(AB)\subseteq D(A)$, so $D(AB)D(A)=D(A)D(AB)=D(AB)$.  Dually for $R$.
\end{proof}

Of course, the partial order on $D(C)$ can be taken to be equality.  In that case, $D,R$ are defined exactly as for the category case given in Example 3 in \cite{Lawson21}.

\section{Introducing the ample conditions for DR-semigroups}

Throughout this section, $S$ is a fixed DR-semigroup.

Following the work of Fountain and others, including Stein in \cite{Stein22}, we say that $S$ satisfies the {\em ample conditions} if it obeys the following two laws: for all $x\in S$ and $e\in D(S)$,
\begin{itemize}
\item $xe=D(xe)x$ (left ample condition), and 
\item $ex=xR(ex)$ (right ample condition).
\end{itemize}
A (left) restriction semigroup is nothing but a DRC-semigroup satisfying the (left) ample condition(s).  (We do not use the terminology ``DR-ample semigroup", or ``ample DR-semigroup", since ample semigroups satisfy the congruence conditions when viewed as biunary semigroups.)

We recall Stein's generalized ample conditions as used in his sequence of papers \cite{Stein22}, \cite{Stein22a} and \cite{Stein24}; these are the laws
\begin{itemize}
\item $D(R(D(xy)x)y) = R(D(xy)x)$ (generalized left ample condition), and 
\item $R(x D(yR(xy))) = D(yR(xy))$ (generalized right ample condition).
\end{itemize}

The first and third parts of the next result were proved within DRC-semigroups in \cite{Stein22} (see Lemmas 3.8 and 3.9 as well as Proposition 3.14 and its dual).

\begin{pro}   \label{several}
Suppose $S$ satisfies the ample conditions. 
\begin{enumerate}
\item $S$ satisfies the generalized ample conditions.
\item The partial orders $\leq_r$ and $\leq_l$ coincide; we call this partial order the {\em standard order} on $S$, and denote it by ``$\leq$" (as it coincides with the natural order on $D(S)$).  
\item $D(S)$ commutes.
\item Multiplication in $S$ is monotonic: 
$s_1\leq t_1\: \&\: s_2\leq t_2\; \Rightarrow\; s_1s_2\leq t_1t_2.$
\end{enumerate}
\end{pro}
\begin{proof}
Now for all $x,y\in S$, $R(xy)=R((xD(y))y)\leq R(R(xD(y))y)$ by (2) in Lemma \ref{DRbits}, so $R(xy)=R(R(xD(y))y)R(xy)$.  Hence,
$$R(xD(y))yR(xy)=yR(R(xD(y))y)R(xy)=yR(xy).$$  
As this is true for all $x,y$, we may replace $y$ by $yR(xy)$ and obtain 
$$R(xD(yR(xy)))yR(xy)R(xy)=yR(xy)R(xy),$$ 
or $R(xD(yR(xy)))yR(xy)=yR(xy)$, and so by (2) in Lemma \ref{DRbits} and Law (DR3) for DR-semigroups,
$$D(yR(xy))\leq R(xD(yR(xy)))\leq R(D(yR(xy)))=D(yR(xy)),$$ and so $D(yR(xy))=R(xD(yR(xy)))$, which is the generalized right ample condition.  Dualize to obtain the other one.

If $x\leq_r y$ for some $x,y\in S$, then $x=D(x)y$ and $D(x)\leq D(y)$, so $x=yR(D(x)y)$, and  $R(D(x)y)\leq R(y)$ by (2) in Lemma \ref{DRbits}, so $x\leq_l y$.  The converse follows dually.

Next, for $e,f\in D(S)$, by (DR4) and the ample conditions, $D(ef)=D(ef)D(e)=D(ef)e=eD(f)=ef$, and so $ef=efe$.  Dually, using $R$ we obtain $fe=efe$, and so $ef=fe$.

Monotonicity follows from Proposition $3.16$ in \cite{Lawson91}, which states that if $D(S)$ commutes and $x\leq_r y$ then $xz\leq_r yz$, and dually for $\leq_l$, so if the partial orders are equal then both hold.
\end{proof}

In fact two of the above properties together characterise the ample conditions.

\begin{pro}
The ample conditions hold on $S$ if and only if $\leq_r,\leq_l$ coincide and $D(S)$ commutes.
\end{pro}
\begin{proof}
The forward direction is covered by Proposition \ref{several}.  Conversely, suppose $\leq_r$ and $\leq_l$ coincide and $D(S)$ commutes.  Then by Propostion 3.13 in \cite{Lawson91}, for all $e\in D(S)$ and $s\in S$, 
$se\leq_l s$, and so $se\leq_r s$, so $se=D(se)s$.  The other ample condition follows dually.
\end{proof} 

Following \cite{Lawson21}, we say $s\in S$ is {\em bideterministic} if for all $e\in D(S)$, $se=D(se)s$ and $es=sR(es)$.  Denote by $B(S)$ the bideterministic elements of $S$.  Every element of $S$ is bideterministic if and only if $S$ satisfies the ample conditions.  More generally, we have the following.

\begin{pro}
The subset $B(S)$ is a subsemigroup of $S$.
\end{pro}
\begin{proof}
Suppose $a,b\in B(S)$, with $e\in D(S)$.  Then $abe=aD(be)b=D(aD(be))ab$, so 
$abe=D(abe)abe=D(abe)D(aD(be))ab=D(abe)ab$ by (2) in Lemma \ref{DRbits}.  Dualising gives $eab=abR(eab)$.  Hence $ab\in B(S)$.
\end{proof} 

Obviously, if $B(S)$ is closed under $D,R$, it is a DR-subsemigroup satisfying the ample conditions.  One time this happens is when $D(S)\subseteq B(S)$.

\begin{pro}  \label{BDS}
The set of projections $D(S)$ commutes if and only if $D(S)\subseteq B(S)$, in which case $B(S)$ is a DR-subsemigroup of $S$ which satisfies the ample conditions.
\end{pro}
\begin{proof}
Suppose $D(S)\subseteq B(S)$.  The argument that $D(S)$ commutes is very similar to the proof given in the third part of Proposition \ref{several}.

Conversely, suppose $D(S)$ commutes.  Then $S$ is a left C-semigroup in the sense of \cite{JackSto01} and so $D(S)$ forms a semilattice by Proposition 1.2 there, so for all $e,f\in D(S)$, $D(ef)=ef$ and so $ef=efe=D(ef)e$, and dually, $ef=fR(ef)$, so $e\in B(S)$.
\end{proof}

Small examples show that $B(S)$ being closed under $D$ and $R$ is not equivalent to $D(S)$ commuting.

\section{Examples of DR-semigroups satisfying the ample conditions}

Aside from restriction semigroups (hence, inverse semigroups), there are other quite natural examples of DR-semigroups satisfying the ample conditions.  We consider two such classes here.

\subsection{Example: closure operators on sets}

Let $X$ be a non-empty set equipped with a closure operator $C$ on its subsets.  When equipped with intersection $\cap$, the power set $2^X$ of $X$ is a semigroup, which is a DR-semigroup in which $D(T)=R(T)=C(T)$ for all $T\subseteq X$.  Moreover, it trivially satisfies the ample conditions, since $(2^X,\cap)$ is a semilattice, but generally does not satisfy the congruence conditions.

\subsection{Example: strongly order-preserving partial functions on a quasiordered set}

Throughout this subsection, suppose $(X,\leq)$ is a quasiordered set.

For any $T\subseteq X$, we define its closure $C(T)$ to be the down-set generated by $T$, that is, the set $\{y\in X\mid y\leq x\mbox{ for some } x\in T\}$.  It is well-known that $C$ is an additive closure operator.


We say a partial function $f$ on $X$ is {\em strongly order-preserving} if, for all $x,y\in \mbox{dom}(f)$, $x\leq y\; \Leftrightarrow\; f(x)\leq f(y)$. Define $P_{SO}(X)$ to be the subset of all strongly order-preserving partial functions on $X$.

It is easy to see that every $f\in P_{SO}(X)$ is injective if the quasiorder $\leq$ is a partial order.  In general, $P_{SO}(X)$ is closed under composition and so forms a semigroup $(P_{SO}(X),\cdot)$ under that operation.  

Define domain and range operations on $P_{SO}(X)$ as follows: for $f\in P_{SO}(X)$, $D(f)$ is the restriction of the identity map to $C(\mbox{dom}(f))$, the down-set generated by $\mbox{dom}(f)$.  Define $R(f)$ dually using $\mbox{im}(f)$.  It follows from Lemma 12.1 in \cite{JackSto13} that $(P_{SO}(X),\cdot,D,R)$ is a DR-semigroup in which $D(P_{SO}(X))$ commutes, which moreover is {\em translucent} in the sense of that reference, meaning that $se=D(se)s$ for all $s\in P_{SO}(X)$ and $e\in D(P_{SO}(X))$; this is nothing but the left ample condition.  In fact we have the following.

\begin{pro}
If $(X,\leq)$ is a quasiordered set, then $P_{SO}(X)$ satisfies the ample conditions.
\end{pro}
\begin{proof}
It remains to check that for all $s\in P_{SO}(X)$ and $e\in D(P_{SO}(X))$, we have $es=sR(es)$.  Recall the inclusion order on partial functions: $f\subseteq g$ if and only if $(x,y)\in f\;\Rightarrow\; (x,y)\in g$.  

Now, $es=esR(es)\subseteq sR(es)$ since $esR(es)$ is a domain restriction of $sR(es)$.  Conversely, if $(x,xs)\in sR(es)$, then $(xs,xs)\in R(es)$, and so $xs\leq y$ for some $y\in \mbox{im}(es)$, so $y=x_1s$ where $x_1\in \mbox{dom}(es)=\mbox{dom}(e)\cap \mbox{dom}(s)$, so $xs\leq x_1s$ and so $x\leq x_1$ since $s\in P_{SO}(X)$, and so $x\in \mbox{dom}(e)$ because it is a down-set, so $x\in \mbox{dom}(e)\cap \mbox{dom}(s)=\mbox{dom}(es)$.  Hence $\mbox{dom}(sR(es))\subseteq \mbox{dom}(es)$, and so $sR(es)\subseteq es$ since both are partial functions contained in $s$.  Hence, $sR(es)= es$.
\end{proof}

Such examples need not satisfy the congruence conditions since they contain a copy of the power set example given earlier: the power set of $X$ arises as the restrictions of the identity to subsets of $X$ and then $D,R$ agree with $C$ as defined above on this copy of the power set of $X$.

\subsection{Examples obtained from partial categories, including injective partial functions}  

If $C$ is a partial category in which $D(C)$ is partially ordered, then by Proposition \ref{BDS}, the set of bideterministic elements of $P(C)$, $B(C)=B(P(C))$, forms a DR-subsemigroup satisfying the ample conditions and containing $D(P(C))$ (the down-sets in $D(C)$).  There is interest in exactly what the bideterministic elements are in this case.

\begin{pro}  \label{bidet}
Let $C$ be a partial category in which $D(C)$ is partially ordered.  Then $A\in B(C)$ if and only if, for all $a,b\in A$, $D(a)\leq D(b)$ if and only if $R(a)\leq R(b)$.  
\end{pro}
\begin{proof}
Suppose $A\in B(C)$.  Pick $a,b\in A$.  Suppose $D(a)\leq D(b)$.  Now let $E=\{e\in D(C)\mid e\leq R(b)\}$.  Then $E=D(E)$, and so because $A\in B(C)$, $AE=D(AE)A$.  Hence, as $b=b\circ R(b)\in AE$, we have $b\in D(AE)A$, so $b=e\circ c$ for some $e\in D(AE), c\in A$, meaning that $e=D(c)$, so $b=D(c)\circ c=c$, and so $e=D(c)=D(b)$.  Since $D(a)\leq D(b)$, we have $D(a)\in D(AE)$ and therefore $a=D(a)\circ a\in D(AE)A=AE$. Hence, $a=a'\circ f$ for some $a'\in A$ and $f\in E$, and so $a=a', f=R(a)\in E$ and so by definition of $E$, $R(a)\leq R(b)$.  Dualising gives the converse.

Now suppose $A$ is such that, for all $a,b\in A$, $D(a)\leq D(b)$ if and only if $R(a)\leq R(b)$.  Now for all $E\in D(P(C))$, $AE=D(AE)AE\subseteq D(AE)A$; we show the opposite inclusion.  Suppose $x\in D(AE)A$.  Then $x=e\circ a$ for some $a\in A$ and $D(a)=e\leq D(a'\circ f)=D(a')$ for some $a'\in A$ and $R(a')=f\in E$.  But since $D(a)\leq D(a')$, we have $R(a)\leq R(a')$, so $R(a)\in E$ as $E$ is a down-set.  Hence, $x=e\circ a=a=a\circ R(a)\in AE$.  So $AE=D(AE)A$.  Dualising gives $EA=A\cdot R(EA)$.
\end{proof}

Boolean Ehresmann monoids are defined in \cite{Lawson21}.  Examples include $P(C)$ where $C$ is a category.  The author defined the notion of a {\em partial isometry} in a Boolean Ehresmann monoid, and it was shown that $A\in P(C)$ is a partial isometry if and only if, for $a,b\in A$, if $D(a)=D(b)$ then $a=b$, and dually in terms of $R$.  In \cite{Lawson21}, it was shown that ${\mathcal PI}(C)$, the set of all partial isometries in $P(C)$, is an Ehresmann subsemigroup of $P(C)$ which is in fact restriction, and moreover every restriction semigroup $S$ embeds in ${\mathcal PI}(C)$ for some category $C$ ($C$ can be chosen to be the derived ordered category of $S$).  

Following \cite{Lawson21}, for a partial category $C$ in which $D(C)$ is partially ordered, we say $A\in P(C)$ is a {\em partial isometry} if $A\in B(C)$, and for all $a,b\in A$, if $D(a)=D(b)$ (or equivalently, $R(a)=R(b)$), we have $a=b$.   If $C$ is a category and the partial order on $D(C)$ is equality, this agrees with Lawson's notion.  Generally, let ${\mathcal PI}(C)$ be the set of all partial isometries; trivially, ${\mathcal PI}(C)\subseteq B(C)$.

\begin{pro}
If $C$ is a partial category in which $D(C)$ is partially ordered, then ${\mathcal PI}(C)$ is a DR-subsemigroup of $B(C)$ and hence satisfies the ample conditions.
\end{pro}
\begin{proof} 
It is obvious that $D(C)\subseteq {\mathcal PI}(C)$, so it will be closed under $D$ and $R$.  If $A,B\in {\mathcal PI}(C)$, then $AB\in B(C)$ as the latter is a subsemigroup; but also, if $a_1\circ b_1,a_2\circ b_2\in AB$ (with $a_1,a_2\in A, b_1,b_2\in B$), and $D(a_1\circ b_1)=D(a_2\circ b_2)$, then $D(a_1)=D(a_2)$, so $a_1=a_2$, so $D(b_1)=R(a_1)=R(a_2)=D(b_2)$, and so $b_1=b_2$, so $a_1\circ b_1=a_2\circ b_2$.  Hence, $AB\in {\mathcal PI}(C)$.
\end{proof}

Again, the partial order on $D(C)$ can be taken to be equality.  The greater generality used here is to enable a representation theorem for DR-semigroups satisfying the ample conditions in Section \ref{sec:rep}.

An example that is not a restriction semigroup may be obtained by letting $C$ consist of the posetal category obtained from the integers described after Example \ref{eg:sat}, so that $C=\{(x,y)\mid x,y\in {\mathbb Z},x\leq y\}$, with $D((x,y))=(x,x), R((x,y))=(y,y)$ and $(x,y)\circ (u,v)=(x,v)$ exists if and only if $y=u$.  Then as noted there, defining $P_n=\{(x,y)\mid y-x<n\}$ for some positive integer $n$, setting $(x,y)\cdot(y,z)$ to be undefined if $z-x\geq n$ and otherwise to equal $(x,z)$ makes $S=(S,\cdot,D,R)$ into a partial category.  Equip $D(P_n)=D(C)$ with the partial order of equality.  Then the elements of ${\mathcal PI}(S)$ may be identified with the injective partial functions $f$ defined on ${\mathbb Z}$ for which $0\leq xf-x<n$ for all $x\in \mbox{dom}(f)$, with $fg$ considered undefined at $x$ if $xfg-x\geq n$.  For example, for $f=\{(1,3), (3,7)\}, g=\{(3,5), (7,9)\}\in S_5$, we have $fg=\{(1,5)\}$, so $D(fg)=\{(1,1)\}$, whereas $D(fD(g))=D(fR(f))=D(f)=\{(1,1),(3,3)\}$, witnessing the failure of the congruence conditions.  

Other examples can be obtained from saturated subsets of categories as in Example \ref{eg:sat}.

\section{Partial products and an ESN theorem}

\subsection{Background on ESN theorems}

ESN-style theorems seek to describe certain classes of semigroups in terms of categories or similar partial structures.  These originated with the work of Ehresmann, Schein and Nambooripad on inverse and regular semigroups.  Work has continued on multiple fronts, but here we are interested in approaches that start with a semigroup equipped with two unary operations, modelling ``domain" and ``range" and which can then be captured via a category or similar partial structure, generally also equipped with one or more partial order.

The original ESN theorem involved inverse semigroups.  This was generalized by Lawson to Ehresmann semigroups in \cite{Lawson91}, and more recently to DRC-semigroups in \cite{SWang22} and \cite{East24}.  (Previously, several intermediate cases were considered: as examples, see \cite{SWang19} and \cite{SWang20}.)  However, so far no ESN-style theorems of this type have been developed for classes of DR-semigroups not satisfying at least one of the two congruence conditions (generally, both are required, but see \cite{StokesESN23} where only one is).

Every inverse semigroup is an Ehresmann semigroup, in which $D(x)=xx', R(x)=x'x$ for all $x$ (where $s'$ is the inverse of $x$); using this, Lawson extended the ESN theorem for inverse semigroups to Ehresmann semgroups in \cite{Lawson91}.  On any Ehresmann semigroup $(S,\cdot,D,R)$, he defined the restricted product given by $s\circ t=st$ but only when $R(s)=D(t)$; in this way, $(S,\circ,D,R)$ is a category.  On its own, the information in this category is not sufficient to capture the structure of the Ehresmann semigroup $S$, but if the partial orders $\leq_r,\leq_l$ defined above for DR-semigroups are ``remembered" too, then this is sufficient.  To see this, first observe that products of the form $es,sf$ where $e,f\in D(S)$ and $e\leq D(s), f\leq R(s)$ may be described purely in terms of the category together with the two partial orders: 
\begin{itemize}
\item $es$ is the unique $t\leq_r s$ for which $D(t)=e$ (the ``restriction of $s$ to $e$"), and
\item $sf$ is the unique $t\leq_l s$ for which $R(t)=f$ (the ``co-restriction of $s$ to $f$").
\end{itemize}
We call $(S,\circ,D,R,\leq_r,\leq_l)$ the {\em derived biordered category} obtained from the Ehresmann semigroup $S$.  
Within it, $\leq_r,\leq_l$ coincide on $D(S)$ which is a meet-semilattice under this partial order with meet equal to semigroup product.  One then observes that for all $x,y\in S$,
$$xy=(xe)\circ (ey),\mbox{ where }e=R(x)D(y),$$
meaning that the original semigroup product and hence the full Ehresmann semigroup structure of $S$ is captured by its derived biordered category $(S,\circ,D,R,\leq_r,\leq_l)$.

Lawson in \cite{Lawson91} was able to write down a relatively small number of first-order properties of the derived biordered category of an Ehresmann semigroup, sufficient to fully characterise the biordered categories $(C,\circ,D,R,\leq_r,\leq_l)$ that arise in this way from Ehresmann semigroups.  We refer the reader to \cite{Lawson91} for the details.  
The inverse semigroup case can then be seen as a special case, in which the two partial orders coincide and the derived singly ordered category is an ordered groupoid.  An intermediate case is restriction semigroups, which are precisely Ehresmann semigroups in which the two partial orders coincide; this case is covered in \cite{Lawson91}.

Various authors have attempted to parallel Lawson's results within other classes of biunary semigroups.  A natural larger class where one might hope this could succeed is DRC-semigroups, since they satisfy the congruence conditions and so determine a category: an Ehresmann semigroup is nothing but a DRC-semigroup $S$ in which $D(S)$ commutes.  
There is again a derived biordered category of a DRC-semigroup.  However, $D(S)$ is no longer a semilattice or even a subsemigroup, but can be made into an algebraic structure with two binary operations given by 
$$e\times f=D(ef),\; e*f=R(ef),\ e,f\in D(S).$$
These give $D(S)$ the structure of a {\em projection algebra}; see \cite{SWang22} and \cite{East24}, where ESN-style theorems are presented.  In the former case, a notion of generalized category is required, in which partial products exist more often than just when $R(x)=D(y)$.  In the latter case, categories are retained, but at the cost of the introduction of a necessary additional layer of structure: a particular functor from a category defined using the projection algebra to the derived biordered category must be specified, to allow the capturing of arbitrary products of projections.  

The ESN-style theorems for Ehresmann semigroups and indeed DRC-semigroups are generalizations of the inverse semigroup case to DR-semigroup settings in which the congruence conditions hold.  In what follows, we generalize the case of inverse semigroups (and indeed restriction semigroups) in a different way: we retain the ample conditions but drop the congruence conditions.

\subsection{The (cat,trace)-product and partial categories}

Throughout this subsection, let $S$ be a fixed DR-semigroup.

It is convenient to define the following predicates on $S$:
\begin{itemize}
\item $cat(x,y)\; \Leftrightarrow\; R(x)=D(y)$
\item $trace(x,y)\; \Leftrightarrow\; D(xy)=D(x)\: \&\: R(xy)=R(y)$.
\end{itemize}

Define the {\em cat-product} on $S$ as before, namely
$$x\circ y=xy\mbox{ but only when } cat(x,y).$$
The resulting structure $(S,\circ,D,R)$ is evidently a category if and only if $S$ satisfies the cat-semigroup condition as in \cite{StokesESN23}, which we may state as ``$cat(x,y)\Rightarrow trace(x,y)$". 

Thus, in seeking ESN theorems for DR-semigroups that are not congruence, one approach would be to work with those at least satisfying the cat-semigroup condition, since then we get a category from the cat-product; however, it is not clear how to then obtain ESN-style theorems in general.  For example, in the Ehresmann case, one has that $xy=xD(y)\circ R(x)y$, where $\circ$ is the cat-product, and we have that $xD(y)=x(R(x)D(y))$ is a corestriction of $x$ and dually for $R(x)y$, with $R(x)D(y)$ the meet of $R(x),D(y)$ in $D(S)$.  One might imagine that such an approach could work in DR-semigroups in which $D(S)$ commutes and hence is a subsemilattice.  However, requiring $xD(y)\circ R(x)y$  to exist for all $x,y$ is shown in \cite{StokesESN23} to imply the congruence conditions, and we are back to Ehresmann semigroups!

We note in passing that an approach that drops orderings entirely and concentrates on a ``biaction" of category elements on its identities is used in \cite{FitzKin21}, \cite{Lawson21} and \cite{StokesESN23}.  But such approaches lack the conceptual simplicity of order-based aporoaches since all products of elements of the category with its identities must be recorded, rather than just two partial orders. Moreover, the congruence conditions are required in \cite{FitzKin21} and \cite{Lawson21}, and at least one of them in \cite{StokesESN23}.

Another option is to consider other partial products and aim for something sufficently category-like yet which may not yield a category; here, the work of Stein in \cite{Stein24} is relevant.  To find a more general setting for his earlier results relating to a problem in representation theory, Stein defined a partial product on a DR-semigroup that may not give rise to a category, but may still yield a partial algebra with enough category-like features for his purposes.   Thus, Stein defined on the DR-semigroup $S$,
$$x\circ y=xy\mbox{ but only when both }cat(x,y)\mbox{ and } trace(x,y).$$
We shall call this the {\em (cat,trace)-product} in what follows.  

Let $\circ$ be the (cat,trace)-product on $S$. If $x\circ y$ exists then $R(x)=D(y)$ (but not necessarily conversely!), and then also $D(x\circ y)=D(x)$, and $R(x\circ y)=R(y)$, and the laws $D(x)\circ x=x\circ R(x)=x$ are easily seen to hold.  However, $(S,\circ)$ may not be a partial semigroup.  

In \cite{Stein24}, Stein found a suitable sufficient condition on the DR-semigroup $S$, not necessarily satisfying the congruence conditions, to ensure that the derived structure $(S,\circ,D,R)$ is a partial category, or equivalently, to ensure that $(S,\circ)$ is a partial semigroup.  His condition, which was that one of his generalized ample conditions hold (see Lemma 3.8 in \cite{Stein24}), was dictated by representation theory considerations: when is the semigroup algebra of a DR-semigroup isomorphic to the algebra obtained from its derived partial algebra under the (cat,trace)-product? 

Of course, another way to ensure that $(S,\circ,D,R)$ is a partial category is to assume that $S$ satisfies the cat-semigroup condition: then, the (cat,trace)-product is simply the cat-product, and $(S,\circ,D,R)$ is a category.  But in general, it is a difficult problem to simply characterise when the (cat,trace)-product $\circ$ on $S$ gives a partial semigroup (so that $(S,\circ,D,R)$ is at least a partial category), since this class includes such diverse classes as DR-semigroups that satisfy the cat-semigroup conditions, and the DR-semigroups satisfying either one of Stein's generalized ample conditions.  In neither of these cases does there seem much opportunity to obtain an ESN-style theorem, at least not in their full generality.

Inverse semigroups and indeed restriction semigroups are Ehresmann semigroups, so of course the cat-semigroup condition $cat(x,y)\Rightarrow trace(x,y)$ holds.  However, the {\em converse} implication holds on these classes also.  Thus, if $S$ is a restriction semigroup, and $trace(x,y)$ holds for some $x,y\in S$, then $D(xy)=D(x),R(xy)=R(y)$, and so $xD(y)=D(xy)x=D(x)x=x$, so $D(y)\leq R(x)$, and similarly, $R(x)\leq D(y)$, and so they are equal, so $cat(x,y)$ holds.  This shows that $S$ satisfies the implication  $trace(x,y)\Rightarrow cat(x,y)$.

In fact it is not difficult to equationally characterise those DR-semigroups in which this implication holds.

\begin{thm}  \label{equiv}
The implication $trace(x,y) \Rightarrow cat(x,y)$ holds on $S$ if and only if $S$ satisfies the ample conditions.
\end{thm}
\begin{proof}
Suppose $trace(x,y)\Rightarrow cat(x,y)$ for all $x,y\in S$.  Now pick $x,y\in S$ and let $a=D(xy)x$, $b=yR(xy)$; then $ab=xy$ and so from Lemma \ref{lemuse}, $D(a)=D(D(xy)x)=D(xy)=D(ab)$ and dually $R(b)=R(yR(xy))=R(xy)=R(ab)$, so $trace(a,b)$, and so $R(D(xy)x)=R(a)=D(b)=D(yR(xy))$.   Hence, for all $e\in D(S)$ and $y\in S$, $D(yR(ey))=R(D(ey)e)=R(D(ey))=D(ey)$ upon using Lemma \ref{DRbits}, and so 
$$yR(ey)=D(yR(ey))yR(ey)=D(ey)eyR(ey)=ey.$$
A dual argument establishes the other ample condition.

Conversely, suppose $S$ satisfies the ample conditions, and $x,y\in S$ are such that $trace(x,y)$ holds.  Then by Lemma \ref{DRbits}, $D(x)=D(xy)\leq D(xD(y))\leq D(x)$, and so $D(xD(y))=D(x)$.  Hence, $xD(y)=D(xD(y))x=D(x)x=x$, so $R(x)\leq D(y)$.  A dual argument establishes that $D(y)\leq R(x)$, and so $cat(x,y)$ holds.
\end{proof}

In fact, the DR-semigroup $S$, satisfying the ample conditions, gives rise to a partial category $(S,\circ,D,R)$ under the (cat,trace)-product because $S$ satisfies the generalized ample conditions as noted in the first part of Proposition \ref{several}, either of which is sufficient as noted above.

\begin{cor}   \label{several2}
Suppose $S$ satisfies the ample conditions, with $\circ$ the (cat,trace)-product on $S$. Then $(S,\circ)$ is a partial semigroup, and so $(S,\circ,D,R)$ is a partial category, which is a category if and only if $S$ satisfies the cat-semigroup condition.  
\end{cor}

Recall that if $S$ is an Ehresmann semigroup, then we may write, for all $x,y\in S$,
$$xy=xD(y)\circ R(x)y,$$
where $\circ$ is the cat-product (equivaently, the (cat,trace)-product).  Here we have the following, which follows from the first half of the proof of Theorem \ref{equiv}.

\begin{pro}  \label{split}
Let $S$ satisfy the ample conditions, with $\circ$ the (cat,trace)-product on $S$.  Then for all $x,y\in S$, $xy=D(xy)x\circ yR(xy)$. 
\end{pro}

Note that $D(xy)\leq D(x), R(xy)\leq R(y)$, and it appears that $D(xy)x, yR(xy)$ can be interpreted as restrictions and corestrictions of $x,y$ respectively.  Thus we appear well on the way to an ESN-style result.  However, some problems exist.

First, there is an issue with defining notions of restriction and corestriction in the absence of the congruence conditions: even if $S$ is a DR-semigroup satisfying the ample conditions, and $s\in S$ and $e\in D(S)$ are such that $e\leq D(s)$, then $D(es)=e$ may fail.  

Another problem is that $D(xy)x$ contains the term ``$xy$".  If the congruence conditions are assumed, then we may write $D(xy)=D(x(R(x)D(y)))$, and then we can express $D(xy)$ in the language of Ehresmann categories, but it is already known how to view Ehresmann (hence restriction) semigroups as Ehresmann categories; see \cite{Lawson91}.    

Fortunately, we do have the following, mitigating the need to consider restrictions and corestrictions entirely.  First, recall the definition of the standard order on the DR-semigroup $S$ satisfying the ample conditions, as in the second part of Proposition \ref{several}.

\begin{pro} \label{corlargest}
Suppose $S$ satisfies the ample conditions.  Then for all $x,y\in S$, $D(xy)x$ is the largest $x_1\in S$ and $yR(xy)$ is the largest $y_1\in S$ (both under $\leq$) for which $x_1\leq x$ and $y_1\leq y$ and for which $x_1\circ y_1$ exists under the (cat,trace)-product.
\end{pro}
\begin{proof}
First, note that $D(xy)x\circ yR(xy)$ exists, and $D(xy)x\leq x, yR(xy)\leq y$.  

Conversely, if $x_1\leq x, y_1\leq y$ and $x_1\circ y_1$ exists, we have $xy_1=xD(y_1)y_1=xR(x_1)y_1=x_1y_1$, and so
$$D(x_1)=D(x_1\circ y_1)=D(x_1y_1)=D(xy_1)=D(xyR(y_1))\leq D(xy),$$
by Lemma \ref{DRbits}, and so $x_1=D(x_1)x=D(x_1)D(xy)x$ with $D(x_1)\leq D(xy)=D(D(xy)xy)\leq D(D(xy)x),$ and so $x_1\leq D(xy)x$.  Dually, $y_1\leq yR(xy)$.
\end{proof}

It is therefore possible to capture the entire structure of the DR-semigroup $S$ satisfying the ample conditions, using the partially ordered partial category $(S,\circ,D,R,\leq)$.  It remains to characterise the latter. 

\subsection{Ample partial categories}

We say the structure $(C,\circ,D,R,\leq)$ is an {\em ample partial category} if 
\begin{enumerate} [ ({APC}1)]
\item $(C,\circ,D,R)$ is a partial category;
\item $(C,\leq)$ is a partially ordered set;
\item $\circ$ is monotonic in $\leq$: if $s_1\leq s_2, t_1\leq t_2$ then $s_1\circ t_1\leq s_2\circ t_2$ if both products exist;
\item $D,R$ are monotonic in $\leq$: $x\leq y$ implies $D(x)\leq D(y)$ and $R(x)\leq R(y)$;
\item for all $x,y\in C$ there are largest $x'\leq x, y'\leq y$ such that $x'\circ y'$ exists;  \label{pseudo}
\item if $x\leq y$ and $D(x)=D(y)$ then $x=y$;  \label{leqD}
\item if $x\leq y$ and $R(x)=R(y)$ then $x=y$;   \label{leqR}
\item if $x\circ y$ exists and $w\leq x\circ y$ then there are (necessarily unique) $x'\leq x, y'\leq y$ such that $w=x'\circ y'$.  \label{below}
\end{enumerate}

Several of these laws are used or at least mentioned in \cite{Lawson91} or \cite{East24}, especially the monotonicity laws.  Thus in the category case, the first four laws above state that $(C,\circ,D,R,\leq)$ is $\Omega$-structured in the sense of \cite{Lawson91}.  Laws (APC\ref{leqD}) and (APC\ref{leqR}) are strengthenings of (OC4) in \cite{Lawson91}, and Law (APC\ref{below}) is (OC7) in \cite{Lawson91}.  We note that most of the laws used in \cite{Lawson91} had their origin in the work of Ehresmann.  One of the laws used in \cite{Lawson91} is that $D(C)$ be a meet-semilattice; this is replaced by (APC\ref{pseudo}) here.

\begin{lem}  \label{lemapc}
If $(C,\circ,D,R)$ is an ample partial category, then for all $x,y,z\in C$:
\begin{enumerate}
\item if $x\leq e\in D(C)$ then $x\in D(C)$;
\item if $x\leq z$ and $y\leq z$ and $D(x)\leq D(y)$ then $x\leq y$;  \label{uniqueD}
\item if $x\leq z$ and $y\leq z$ and $R(x)\leq R(y)$ then $x\leq y$.
\end{enumerate}
\end{lem}
\begin{proof}
Suppose $x\leq e$.  Then there are largest $f\leq D(x)$ and $e'\leq e$ such that $f\circ e'$ exists.  But $D(x)\circ x$ exists with $x\leq e$, so $f=D(x)=D(e')$ and $x\leq e'$, so $x=e'$ by (APC\ref{leqD}).  Now $D(x)=D(e')\leq D(e)=e$, so because $D(x)\circ D(x)$ exists, we must have $D(x)\leq x$, and so because $D(D(x))=D(x)$, we must have $D(x)=x$ upon applying (APC\ref{leqD}); hence $x\in D(C)$.

Suppose $x\leq z, y\leq z$ and $D(x)\leq D(y)$.  Now there are largest $e\leq D(y), z'\leq z$ such that $e\circ z'$ exists (whence $e=D(z')$ since $e\in D(C)$ by the first part).  Since $D(y)\leq D(y), y\leq z$ and $D(y)\circ y$ exists, it must be that $D(y)\leq e\leq D(y)$, so $D(y)=e=D(z')$, and $y\leq z'$, and so by (APC\ref{leqD}), $y=z'$.  Now $D(x)\leq D(y)$, $x\leq z$ and $D(x)\circ x$ exists, so by maximality of $D(y),y$ such that $D(y)\leq D(y)$, $y\leq z$ and $D(y)\circ y$ exists, we must have that $x\leq y$.  Dually for the third property.
\end{proof}

The first law in the above proposition asserts that $D(C)$ forms an order ideal and is called (OI) in \cite{Lawson91}.  Combining this with (APC\ref{pseudo}), this implies that $D(C)$ is a meet-semilattice under $\leq$.
 
From the second and third parts of Lemma \ref{lemapc}, we easily obtain the following.

\begin{cor}  \label{lemapccor}
Let $C$ be an ample partial category.  Then if $x\leq z, y\leq z$ and $D(x)=D(y)$ or $R(x)=R(y)$, then $x=y$.
\end{cor}

The parenthetical uniqueness claim made in Law (APC\ref{below}) above now follows, for if $x'\circ y'=x''\circ y''$ where $x',x''\leq x$ and $y',y''\leq y$, then $D(x')=D(x'\circ y')=D(x''\circ y'')=D(x'')$ and so by the previous corollary, $x'=x''$, and dually $y'=y''$.

Although less useful in the current setting, we may define notions of restriction and corestriction in ample partial categories, the first two of which follow courtesy of (APC\ref{leqD}) and (APC\ref{leqR}) above, the third from (APC\ref{pseudo}).

\begin{pro}
Let $C$ be an ample partial category.
\begin{enumerate}
\item For all $y\in C$ and $e\leq D(y)$, there is at most one $x\leq y$ with $D(x)=e$, the {\em restriction of $y$ by $e$,} denoted by $e|y$;
\item for all $y\in C$ and $e\leq R(y)$, there is at most one $x\leq y$ with $R(x)=e$, the {\em corestriction of $y$ by $e$,} denoted by $y|e$;
\item for all $x,y\in C$ there is a largest $e\in D(C)$ such that $x|e,e|y$ and $x|e \circ e|y$ exist.
\end{enumerate}
\end{pro}
\begin{proof}
The first two claims are immediate from Corollary \ref{lemapccor}.  For the third, picking $x',y'$ as in (APC\ref{pseudo}) and letting $e=R(x')=D(y')$, clearly $x'=x|e, y'=e|y$ and so $x|e \circ e|y$ exist.  If also $x|f,f|y$ and $x|f \circ f|y$ exist then by maximality of $x',y'$ we have $x|f\leq x|e, f|y\leq e|y$, so $f=R(x|f)\leq R(x|e)=e$.
\end{proof}

If $C$ is an ample partial category and $x,y\in C$, we call $e$ in Part 3 of the above the {\em matching term for $x,y$}.

If $C$ happens to be a category, then because $R(x|e)=e=D(e|y)$ if $x|e,e|y$ exist, it follows that the matching term is simply the largest $e\in D(C)$ such that $x|e,e|y$ exist.  In the special case in which $C$ is Ehresmann, $e=R(x)\wedge D(y)$, but not in general (even though this meet exists).

\begin{thm}
If $S$ is a DR-semigroup satisfying the ample conditions, with $\leq$ its standard order and $\circ$ the (cat,trace)-product, then $C(S)=(S,\circ,D,R,\leq)$ is an ample partial category, and $D(C(S))=D(S)$.
\end{thm} 
\begin{proof}
We have already noted that $C(S)$ is a partial category, and that $\leq$ is a partial order, so Laws (APC1) and (APC2) hold.  Law (APC3) follows from the fourth part of Proposition \ref{several}.  Laws (APC4), (APC\ref{leqD}) and (APC\ref{leqR}) come from the definitions of $\leq_r,\leq_l$ (both equalling $\leq$), and Law (APC\ref{pseudo}) comes from Proposition \ref{corlargest}.  

Finally, if $x\circ y$ exists and $w\leq x\circ y=xy$ then $w=D(w)xyR(w)$.  Let $x'=D(w)x, y'=yR(w)$.  Then $x'y'=w$.  But $D(x')=D(D(w)x)\leq D(D(w))=D(w)=D(x'y')\leq D(x')$, so $D(x')=D(x'y')$, and dually, $R(y')=R(x'y')$, so $trace(x',y')$, and so $cat(x',y')$ by Theorem \ref{equiv}, and so $w=x'\circ y'$ exists where $\circ$ is the (cat,trace)-product.  So Law (APC\ref{below}) holds.
\end{proof}

If $S$ is a DR-semigroup satisfying the ample conditions, with $\leq$ its standard order and $\circ$ the (cat,trace)-product, we call  $C(S)=(S,\circ,D,R,\leq)$ its {\em derived ample partial category}.

\subsection{The ESN-style theorem}

We have seen that DR-semigroups satisfying the ample conditions give ample partial categories defined on the same underlying set.  We next show that this process can be run in the other direction, and that the resulting constuctions are mutually inverse.  We go on to show that, with suitably defined morphisms in each, the categories of ample partial categories and of DR-semigroups satisfying the ample conditions are isomorphic.
 
On the ample partial category $C$, define the {\em pseudoproduct} $\otimes$ as follows:
$$x\otimes y=x'\circ y',\mbox{ where $x'\leq x, y'\leq y$ are largest such that $x'\circ y'$ exists},$$
as in (APC\ref{pseudo}) for ample partial categories.  This is a well-defined binary operation on $C$.  Let $S(C)=(C,\otimes,D,R)$.

\begin{thm}
For any ample partial category $C$, $S(C)$ is a DR-semigroup satisfying the ample conditions, and $D(S(C))=D(C)$.
\end{thm}
\begin{proof}
We begin with associativity.  Consider $(s\otimes t)\otimes u$ (where $s,t,u\in C$).  Then this equals $(s'\circ t')'\circ u'$, where $s', t'\in C$ are largest such that $s'\leq s, t'\leq t$ and $s'\circ t'$ exists, and then $(s'\circ t')'\leq s'\circ t'$ and $u'\leq u$ are largest such that $(s'\circ t')'\circ u'$ exists. By (APC\ref{below}), there are $s''\leq s'\leq s, t''\leq t'\leq t$ such that $(s'\circ t')'=s''\circ t''$.  Evidently, $s''\circ t''\circ u'$ exists (recalling that this is unambiguous).  

Now suppose also $s_1\circ t_1\circ u_1$ exists, where $s_1\leq s, t_1\leq t, u_1\leq u$.  By definition of $s',t'$, we have $s_1\leq s', t_1\leq t'$, so $s_1\circ t_1\leq s'\circ t'$ by (APC3).  Hence, since $(s_1\circ t_1)\circ u_1$ exists, we must have that $s_1\circ t_1\leq (s'\circ t')'=s''\circ t''$ and $u_1\leq u'$.  So, $D(s_1)\leq D(s'')\leq D(s)
$, and so $s_1\leq s''$ by \ref{uniqueD} in Lemma \ref{lemapc}, and dually $t_1\leq t''$.  It follows that $s'',t'',u'$ are largest subject to $s''\leq s, t''\leq t, u'\leq u$ and $s''\circ t''\circ u'$ existing.  By symmetry, $s\otimes(t\otimes u)$ must equal the product of these three largest elements subject to these constraints as well.

That $D(s)\circ s$ exists and equals $s$ shows that $D(s)\otimes s=D(s)\circ s=s$; dually, $s\otimes R(s)=s$.  Of course, $R(D(s))=D(s), D(R(s))=R(s)$ for all $s$.  For $e,f\in D(C)$, (APC\ref{pseudo}) gives that $e\otimes f$ is the meet of $e,f$ under $\leq$, and so $e\leq f$ if and only if $f\otimes e=e\otimes f=e$, and so because $D(s\otimes t)=D(s'\circ t')=D(s')$ for some $s'\leq s, t'\leq t$, we must have that $D(s\otimes t)\leq D(s)$ and dually for $R$, so $D(s\otimes t)\otimes D(s)=D(s)\otimes D(s\otimes t)=D(s\otimes t)$.  Hence, $S(C)$ is a DR-semigroup, and obviously $D(S(C))=D(C)$.  

It remains to check that $S(C)$ satisfies the ample conditions.  But for $s\in C, e\in D(C)$, $s\otimes e=s'\circ e'$ for some $s'\leq s,e'\leq e$ (so $e'\in D(C)$ by 1 in Lemma \ref{lemapc}), so $s\otimes e=s'$, whereas $D(s\otimes e)\otimes s=D(s')\otimes s=f\circ s''$ where $f\leq D(s'), s''\leq s$, so $f=D(s'')\leq D(s')$, so $s''\leq s'$ by 2 in Lemma \ref{lemapc}, and so because $D(s')\circ s'$ exists, maximality ensures that $D(s\otimes e)\otimes s=D(s')\circ s'=s'=s\otimes e$.  Dually for the other law.
\end{proof}

\begin{thm}
The constructions $S\mapsto C(S), C\mapsto S(C)$ are mutually inverse.
\end{thm}
\begin{proof}
If $S$ is a DR-semigroup satisfying the ample conditions, the fact that the pseudoproduct on $S(C(S))$ coincides with the product on $S$ is a consequence of Propositions \ref{split} and \ref{corlargest}.  Of course, $D$ and $R$ are unchanged.

Conversely, suppose $(C,\circ,D,R,\leq)$ is an ample partial category.  Define $s*t=s\otimes t$ whenever $cat(s,t)$ and $trace(s,t)$ holds in $(S,\otimes,D,R)$; this is the partial category product in $C(S(C))$.  Now for all $s,t\in C$, $s*t$ exists in $C(S(C))$ if and only if $R(s)=D(t)$, $D(s\otimes t)=D(s)$ and $R(s\otimes t)=R(t)$.  So if $s*t$ exists then writing $s\otimes t=s'\circ t'$ for some $s'\leq s, t'\leq t$ in $C$, we have that $D(s'\circ t')=D(s)$, so $D(s')=D(s)$ and so $s'=s$ by (APC\ref{leqD}) for ample partial categories, and dually $t'=t$, so $s*t=s\otimes t=s\circ t$.  Conversely, if $s\circ t$ exists then obviously $s\otimes t=s\circ t$, so $R(s)=D(t)$, $D(s\otimes t)=D(s)$ and $R(s\otimes t)=R(t)$, and so $s*t$ exists.  Again, $D$ and $R$ are unchanged in passing from $C$ to $C(S(C))$.  

Finally, we must check that the partial order $\leq'$ on $C(S(C))$, namely the one on $S(C)$, agrees with the original one on $C$, $\leq$. Pick $s,t\in C$, and note that $D(s)\otimes t=e\circ t'$ for some $e\leq D(s)$, so $e\in D(C)$, and $t'\leq t$, with $D(t')=e$, so $D(s)\otimes t=t'\leq t$.  First, suppose $s\leq t$ in $C$.  Then $D(s)\leq D(t)$ and so in $D(s)\otimes t=e\circ t'$, by maximality we must have $e=D(s)$ since $D(s)\circ s$ exists and $D(s)\leq D(s), s\leq t$.  Hence $t'\leq t$ and $D(t')=e=D(s)$, so $t'=s$ since $s\leq t$ upon using Corollary \ref{lemapccor}, and so $D(s)\otimes t=D(s)\circ s=s$; hence $s\leq' t$.  Conversely, suppose $s\leq' t$; then $s=D(s)\otimes t=t'\leq t$, so $s\leq t$.  Hence, $\leq$ and $\leq'$ coincide.
\end{proof}

The class of DR-semigroups is a category with morphisms the semigroup homomorphisms respecting $D$ and $R$; within this is the full subcategory of DR-semigroups satisfying the ample conditions.  The class of (small) ample partial categories can be made into a category by taking the morphisms to be ``functors" respecting at least the partial order.  Here, a ``functor" $F:C\rightarrow D$ between partial categories is a mapping that preserves the partial products as well as $D$ and $R$.

However, to obtain a category isomorphism, it is not enough to assume that functors preserve the partial order.  In the case of Ehresmann categories, it was also necessary to assume that the functors preserve meets of identities.  In our case, even this is not sufficient (although the condition we obtain is equivalent to this in the Ehresmann case).

Thus if $C_1,C_2$ are ample partial categories, we say $f:C_1\rightarrow C_2$ is an {\em ample functor} if, for all $s,t\in C_1$,
\begin{itemize}
\item if $s\circ t$ exists then so does $f(s)\circ f(t)$, and $f(s\circ t)=f(s)\circ f(t)$;
\item $f(D(s))=D(f(s))$ and $f(R(s))=R(f(s))$;
\item if $s\leq t$ then $f(s)\leq f(t)$; and
\item if $e$ is the matching term for $s,t$ then $f(e)$ is the matching term for $f(s),f(t)$.
\end{itemize}

\begin{pro}
Suppose $C_1,C_2$ are ample partial categories, with $f:C_1\rightarrow C_2$ an ample functor.  Then for $s,t\in C_1$:
\begin{itemize}
\item if $s',t'$ are the largest $s'\leq s,t'\leq t$ for which $s'\circ t'$ exists, then $f(s'),f(t')$ are the largest $u,v\in C_2$ such that $u\leq f(s)$ and $v\leq f(t)$ and $u\circ v$ exists.
\end{itemize}
\end{pro}
\begin{proof}
Pick $s,t\in C_1$; if $s',t'$ are the largest $s'\leq s,t'\leq t$ for which $s'\circ t'$ exists, then $e=R(s')=D(t')$ is the matching term of $s,t$, and so $g=f(e)$ is the matching term for $f(s),f(t)$.  Hence, $u=f(s)|g, v=g|f(t)$ are the largest $u\leq f(s), v\leq f(t)$ such that $u\circ v$ exists.  Next note that $f(s|e\circ e|t)=f(s|e)\circ f(e|t)$, with $f(s|e)\leq f(s)$, $f(e|t)\leq f(t)$, and $R(f(s|e))=f(R(s|e))=f(e)=g$, so $f(s')=f(s|e)=f(s)|g=u$, and dually $f(t')=v$. 
\end{proof}

The following is a generalisation of Theorem 5.7 in \cite{Lawson91}, where restriction semigroups were called idempotent-connected semigroups and the relevant class of categories was called the class of inductive$_1$-categories.

\begin{thm}
The category of DR-semigroups satisfying the ample conditions is isomorphic to the category of ample partial categories.
\end{thm}
\begin{proof}
Suppose $S,T$ are DR-semigroups satisfying the ample conditions and that $f:S\rightarrow T$ is a homomorphism respecting $D,R$.  If $a,b\in S$ are such that $a\circ b$ exists in $C(S)$, then $cat(a,b)$ and $trace(a,b)$ hold, so applying the homomorphism property of $f$ we obtain $cat(f(a),f(b))$ and $trace(f(a),f(b))$ hold, and so $f(a)\circ f(b)$ exists and evidently equals $f(a)f(b)=f(ab)=f(a\circ b)$, again from the homomorphism property of $f$. Of course, $f$ respects $D$ and $R$.  Moreover, if $a,b\in S$ are such that $a\leq b$ then $D(a)\leq D(b)$, so $D(a)=D(a)D(b)$ and so $D(f(a))=f(D(a))=f(D(a)D(b))=D(f(a))D(f(b))$, so $D(f(a))\leq D(f(b))$, but also $a=D(a)b$, so $f(a)=D(f(a))f(b)$, so $f(a)\leq f(b)$.  So $f$ can be viewed as a functor $C(S)\rightarrow C(T)$ that respects $\leq$.  Furthermore, we know that $x'=D(xy)x=xR(D(xy)x)$ and $y'=yR(xy)=D(yR(xy))y$ are the largest $x'\leq x, y'\leq y$ for which $x'\circ y'$ exists by Proposition \ref{corlargest}, and that the matching term for $x,y$ is $e=R(D(xy)x)=D(yR(xy))$; similarly, the matching term for $f(x),f(y)$ in $C(T)$ is $$R(D(f(x)f(y))f(x))=R(D(f(xy))f(x))=R(f(D(xy)x))=f(R(D(xy)x))=f(e).$$ It now follows that we have a functor $F$ from the category of DR-semigroups satisfying the ample conditions to the category of ample partial categories, in which $F(S)=C(S)$ for any DR-semigroup $S$, and for the homomorphism $f:S\rightarrow T$, $F(f):C(S)\rightarrow C(T)$ is defined by setting $F(f)=f$ as functions.

Conversely, suppose $g:C_1\rightarrow C_2$ is an ample functor.  Of course, $D$ and $R$ are preserved, and for all $s,t\in S$, $s\otimes t=s'\circ t'$, where $s'\leq s,t'\leq t$ are largest such that $s'\circ t'$ exists.  Then $g(s\otimes t)=g(s'\circ t')=g(s')\circ g(t')=g(s)\otimes g(t)$ by the previous proposition and the definition of $\otimes$.  This shows that we have a functor $G$ from the category of ample partial categories to the category of DR-semigroups satisfying the ample conditions in which $G(C)=S(C)$ and for the ample functor $g:C_1\rightarrow C_2$, $G(g): S(C_1)\rightarrow S(C_2)$ is defined by setting $G(g)=g$ as functions.

The functors $F$ and $G$ are obviously mutually inverse, completing the proof.
\end{proof}

\section{An embedding theorem}  \label{sec:rep}

We can now give a representation theorem for DR-semigroups satisfying the ample conditions, generalizing Theorem 3 in \cite{Lawson21}.  

\begin{thm}
Let $S$ be a DR-semigroup satisfying the ample conditions. Then there is a partial category $C$ in which $D(C)$ is partially ordered and an injective morphism $f:S \rightarrow {\mathcal PI}(C)$.
\end{thm}
\begin{proof}
As usual, denote by $\leq$ the standard order on $S$.  Let $C=C(S)$ be the derived ample partial category of $S$; then in particular $D(C)$ is partially ordered by $\leq$ and so $P(C)$ may be constructed.  Define $f:S\rightarrow P(C)$ by setting $f(s)=\{t\in S\mid t\leq s\}$, for all $s\in S$.  

We first show that $f$ is a semigroup homomorphism.  Suppose $a,b,c\in S$ are such that $c\in f(ab)$, so that $c\leq ab$.  Then $c=D(c)ab$ with $D(c)\leq D(ab)\leq D(a)$ by the second part of Lemma \ref{DRbits}.  Hence, because $xy=D(xy)x\circ yR(xy)$ using the (cat,trace)-product, we have
$$c=(D(c)a)b=D(D(c)ab)D(c)a\circ bR(D(c)ab)=D(c)a\circ bR(c)=a'\circ b',$$
where $a'\leq a,b'\leq b$, and so $c\in f(a)f(b)$.  Hence, $f(ab)\subseteq f(a)f(b)$.

Conversely, suppose $c\in f(a)f(b)$, so $c=a'\circ b'$ for some $a'\leq a,b'\leq b$.  But $a'\circ b'=a'b'\leq ab$ by (4) in Proposition \ref{several}, and so $c\in f(ab)$.  Hence, $f(ab)= f(a)f(b)$. 

Next, we show that $f$ respects $D$ and $R$.  But if $e\in D(f(a))$ then $e\leq D(b)$ for some $b\leq a$, so $e\leq D(b)\leq D(a)$ and so $e\in f(D(a))$.  Conversely, if $e\in f(D(a))$ then $e\leq D(a)$ and so obviously $e\in D(f(a))$.  Hence $D(f(a))=f(D(a))$.  Dually for $R$.

Pick $s\in S$.  Suppose $a,b\in f(s)$, meaning that $a\leq s, b\leq s$, so $a=D(a)s$, $b=D(b)s$, $D(a)\leq D(s)$ and $D(b)\leq D(s)$.  If $D(a)\leq D(b)$, then $a=D(a)s=D(a)D(b)s=D(a)b$ also, so $a\leq b$ and so $R(a)\leq R(b)$.  Dualise to get the converse.  Moreover, if $D(a)=D(b)$, then $a=D(a)s=D(b)s=b$. So $f$ embeds $S$ in ${\cal PI}(C)$ as a DR-semigroup.
\end{proof}

\section*{Acknowledgements}

The author is very grateful to the anonymous referee for comments that led to significant improvements in the content and layout of the article, especially the first and third sections.

\vspace{2cm}

\noindent Tim Stokes\\
Mathematics\\
University of Waikato\\
Hamilton 3216\\
New Zealand.\\
email: tim.stokes@waikato.ac.nz

\end{document}